\documentclass{amsart}
\usepackage{amstext,amssymb,amsthm,amsopn,newlfont,graphpap,graphics,graphicx,mathrsfs,enumitem}\usepackage{amstext,amssymb,amsthm,amsopn,newlfont,graphpap,graphics,graphicx,mathrsfs,enumitem,tikz-cd}
\usepackage[parfill]{parskip}
\usepackage[noadjust]{cite}
\usepackage{epigraph}
\usepackage[colorlinks=true,
            linkcolor=red,
            urlcolor=blue,
            citecolor=magenta]{hyperref}
\usepackage{color}
\usepackage{mathrsfs}
\allowdisplaybreaks
\theoremstyle{plain}
\newtheorem{thm}{Theorem}[section]
\newtheorem*{thm*}{Theorem}
\newtheorem{prop}{Proposition}[section]
\newtheorem*{prop*}{Proposition}
\newtheorem{cor}{Corollary}[section]
\newtheorem*{cor*}{Corollary}
\newtheorem{lem}{Lemma}[section]
\newtheorem*{lem*}{Lemma}
\theoremstyle{definition}
\newtheorem{defn}{Definition}[section]
\newtheorem*{defn*}{Definition}

\newtheorem*{exmp*}{Example}
\newtheorem{exmps}{Examples}[section]
\newtheorem*{exmps*}{Examples}

\newtheorem{rem}{Remark}[section]
\newtheorem*{rem*}{Remark}
\newtheorem{rems}{Remarks}[section]
\newtheorem*{rems*}{Remarks}

\newtheorem*{note*}{Note}
\newcommand{\N}{{\mathbb N}}%%%%%%%%%NATURAL NUMBERS
\newcommand{\Z}{{\mathbb Z}}%%%%%%%%%INTEGERS
\newcommand{\R}{{\mathbb R}}%%%%%%%%%REAL NUMBERS
\newcommand{\C}{{\mathbb C}}%%%%%%%%%COMPLEX NUMBERS
\newcommand{\F}{{\mathbb F}}

\renewcommand{\iff}{\: \Leftrightarrow\: }
\renewcommand{\bar}{\overline}

\numberwithin{equation}{section}
\DeclareMathOperator{\orb}{orb}
 
\DeclareMathOperator{\Per}{Per} 
\begin{document}
%%%%%%%%%%%%%%%%%%%%%%TITLE%%%%%%%%%%%%%%%%%%%%%%%%%%%%%%%%%
\title[On linear chaos in the spaces of vanishing and convergent sequences]
{On linear chaos in the spaces of vanishing\\ and convergent sequences}
%%%%%%%%%%%%%%%%%%%%%%%%%%%%%%%%%%%%%%%%%%%%%%%%%%%%%%
% first author
\author[Marat V. Markin]{Marat V. Markin}
%%%%%%%%%%%%%%%%%%%%%ADDRESS%%%%%%%%%%%%%%%%%%%%%%%%%%%%%%%%
\address{
Department of Mathematics\newline
%College of Science and Mathematics\newline
California State University, Fresno\newline
5245 N. Backer Avenue, M/S PB 108\newline
Fresno, CA 93740-8001, USA
}
%%%%%%%%%%%%%%%%%%%%%E-MAIL%%%%%%%%%%%%%%%%%%%%%%%%%%%%%%%%%
\email[corresponding author]{mmarkin@csufresno.edu}
%%%%%%%%%%%%%%%%%%%%%%AUTHORS%%%%%%%%%%%%%%%%%%%%%%%%%%%%%%
% second author
\author{Gabriel Martinez Lazaro}
%\address{}
%%%%%%%%%%%%%%%%%%%%%E-MAIL%%%%%%%%%%%%%%%%%%%%%%%%%%%%%%%%%
\email{gmartinez\_laz@mail.fresnostate.edu}
%%%%%%%%%%%%%%%%%%%%%%AUTHORS%%%%%%%%%%%%%%%%%%%%%%%%%%%%%%
% third author
\author{Edward S. Sichel}
%\address{}
%%%%%%%%%%%%%%%%%%%%%E-MAIL%%%%%%%%%%%%%%%%%%%%%%%%%%%%%%%%%
\email{edsichel@mail.fresnostate.edu}
%%%%%%%%%%%%%%%%%%%%%DEDICATORY%%%%%%%%%%%%%%%%%%%%%%%%%%%%%
%\dedicatory{}
%%%%%%%%%%%%%%%%%%%%DATE%%%%%%%%%%%%%%%%%%%%%%%%%%%%%%%%%%%%
%\date{}
%%%%%%%%%%%%%%%%%%%%%ACKNOWLEDGEMENTS%%%%%%%%%%%%%%%%%%%%%%%
%\thanks{}
%%%%%%%%%%%%%%%%%%%SUBJECT CLASSIFICATION%%%%%%%%%%%%%%%%%%
\subjclass{Primary 47A16, 47B37; Secondary 47A10}
%%%%%%%%%%%%%%%%%%%%KEYWORDS%%%%%%%%%%%%%%%%%%%%%%%%%%%%%%%%
\keywords{Hypercyclic vector, periodic point, hypercyclic operator, chaotic operator, spectrum}
%%%%%%%%%%%%%%%%%%%%ABSTRACT%%%%%%%%%%%%%%%%%%%%%%%%%%%%%%%%
\begin{abstract}
We study the chaoticity of bounded and unbounded weighted backward shifts in the space $c_0(\N)$ of vanishing sequences via a novel straightforward approach based on a newly found \textit{sufficient condition for linear chaos} and show that their extensions to the space $c(\N)$ of convergent sequences are not even hypercyclic. Thus, we furnish bounded and unbounded linear chaotic operators in $c(\N)$ in a different way: as conjugates to the weighted backward shifts in $c_0(\Z_+)$ via a homeomorphic isomorphism between the two spaces.
\end{abstract}
%%%%%%%%%%%%%%%%%%%%%%%%%%%%%%%%%%%%%%%%%%%%%
\maketitle
%%%%%%%%%%%%%%%%%%%%%%%%%EPIGRATH%%%%%%%%%%%%%%
\epigraph{\textit{It turns out that an eerie type of chaos can lurk just behind a facade of order - and yet, deep inside the chaos lurks an even eerier type of order.
}}{Douglas R. Hofstadter}

%%%%%%%%%%%%%%%%%%%%%%%%%%%%%%%%%%%%%%%%
\section[Introduction]{Introduction}\label{intro}

We study the chaoticity of the weighted backward shifts in the space 
\[ 
c_0(\N):= \left\{ x:= (x_k)_{k \in \N}\in \F^\N \,\middle|\, \lim_{k \to \infty}x_k = 0 \right\}
\]
($\N:=\left\{1,2,3,\dots\right\}$ is the set of \textit{natural numbers}, $\F:=\R$ or $\F:=\C$) of $\F$-termed vanishing sequences, which are \textit{bounded}
\[
c_0(\N)\ni x:=\left(x_k\right)_{k\in \N}\mapsto A_wx:=w\left(x_{k+1}\right)_{k\in \N}\in c_0(\N)\quad (|w|>1),
\]
introduced in \cite{Rolewicz} (see also \cite{Godefroy-Shapiro1991}), or \textit{unbounded} 
\[
A_wx:=\left(w^kx_{k+1}\right)_{k\in \N} \quad (|w|>1)
\]
with maximal domain
\[
D(A_w):=\left\{ x:=\left(x_k\right)_{k\in \N} \in c_0(\N) \,\middle|\, \left(w^kx_{k+1}\right)_{k\in \N}\in c_0(\N) \right\},
\]
introduced in \cite{arXiv:1811.06640}, via a novel straightforward approach based on a newly found \textit{Sufficient Condition for Linear Chaos} \cite[Theorem $3.2$]{arXiv:2106.14872} not requiring explicit construction of a hypercyclic vector and a dense set of periodic points. 

We furnish concise proofs for the chaoticity of these linear operators along with their powers and analyze their spectral structure.

We further show that the extensions of the aforementioned weighted backward shifts to the space
\[ 
c(\N):= \left\{ x:= (x_k)_{k \in \N}\in \F^\N \,\middle|\, \exists \lim_{k \to \infty}x_k \in \F \right\}
\]
of $\F$-termed convergent sequences are not even hypercyclic.

Thus, we furnish bounded and unbounded linear chaotic operators in $c(\N)$ in a different way: as conjugates to the weighted backward shifts in 
\[ 
c_0(\Z_+):= \left\{ x:= (x_k)_{k \in \Z_+}\in \F^{\Z_+}\,\middle|\,\lim_{k \to \infty}x_k = 0  \right\}
\]
($\Z_+:=\left\{0,1,2,\dots\right\}$ is the set of \textit{nonnegative integers}) via a homeomorphic isomorphism between the two spaces.

\begin{rems}\
\begin{itemize}
\item As follows from the inclusions
\[
c_0(\N)\subset c(\N)\subset l_\infty(\N),
\]
the space $c(\N)$ lives between the space $c_0(\N)$, where linear chaos is known to exist, and the space
\[ 
l_\infty(\N):= \left\{ x:= (x_k)_{k \in \N}\in \F^\N \,\middle|\, \sup_{k \in \N}|x_k| <\infty \right\}
\]
of $\F$-termed bounded sequences, where even hypercyclicity has no place.
%%%%%
\item Henceforth, we use the notations $c_0(\N)$, $c(\N)$ for the spaces of vanishing and convergent sequences over $\N$, respectively, and the notations $c_0(\Z_+)$, $c(\Z_+)$ for their counterparts over $\Z_+$. We also use the shorter notations $c_0$ and $c$ whenever the indexing set is implied contextually.
\end{itemize}
\end{rems}

The chaoticity of the bounded weighted backward shifts in $c_0(\N)$ was first established in \cite{Rolewicz,Godefroy-Shapiro1991}. We reestablish this result in Theorem \ref{BLCc0} via the \textit{Sufficient Condition for Linear Chaos} \cite[Theorem $3.2$]{arXiv:2106.14872} (cf. the original constructive proofs for hypercyclicity  \cite{Rolewicz} and dense periodicity \cite{Godefroy-Shapiro1991}, respectively) and go beyond by proving the chaoticity of all natural powers of these operators and analyzing their spectral structure. In Theorem \ref{ULCc0}, via the aforementioned sufficient condition, we reestablish the chaoticity for the unbounded weighted backward shifts in $c_0(\N)$, first introduced and studied in \cite{arXiv:1811.06640} (cf. the original constructive proof), prove the chaoticity of all their natural powers, and replicate the analysis of their spectral structure provided in \cite{arXiv:1811.06640}.

All the results of Sections \ref{WBSc} and \ref{LCc} are entirely novel.

%The results herein have been published in preprint \cite{MarkM-LSich2022}.

%%%%%%%%%%%%%%%%%%%%%%%%%%%%%%%%%%%%%%%%%%%%%
\section[Preliminaries]{Preliminaries}

The subsequent preliminaries are essential for our discourse.

%\begin{enumerate}
%    \item Denseness; Seperable space - Gabriel
%    \item Schauder basis - Gabriel
%    \item Bourdon's theorem - Gabriel

%\end{enumerate}

%%%%%%%%%%%%%%%%%%%%%%%%%%%%%%%%%%%%%%%%%%%%%
\subsection{Spaces $c_0$ and $c$}\

The spaces $c_0(\N)$ and $c(\N)$ are infinite-dimensional separable Banach spaces relative to $\infty$-norm
	 \[ 
	 x:= (x_k)_{k\in \N} \mapsto \|x\|_\infty := \sup_{k \in \N} |x_k|, 
	 \]
	the former being a closed \textit{hyperplane}, which is a \textit{nowhere dense} subspace, of the latter (see, e.g., \cite{Markin2027EFA,Markin2026EOT}).

The \textit{limit functional}
\begin{equation}\label{lf}
c(\N) \ni x:= (x_n)_{n\in \N} \mapsto l(x):= \lim_{n \to \infty} x_n \in \F,
\end{equation}
is a bounded linear functional on $c(\N)$ with $\ker l = c_0(\N)$ (see, e.g., \cite{Markin2027EFA,Markin2026EOT}).

Relative to the standard Schauder basis $\left\{e_n:=\left(\delta_{nk}\right)_{k\in \N}\right\}_{n\in \N}$ for $c_0(\N)$, where $\delta_{nk}$ is the \textit{Kronecker delta}, each $x := \left(x_k\right)_{k \in \N}\in c_0(\N)$ allows the Schauder expansion
\[
x = \sum_{k=1}^\infty c_k(x) e_k
\]
with the \textit{coordinates} $c_k(x)=x_k$, $k\in \N$.

Relative to the standard Schauder basis $\left\{e_n\right\}_{n\in \Z_+}$ for $c$, where
\[
e_0 := (1,1,1,\dots)\quad \text{and}\quad e_n:=\left(\delta_{nk}\right)_{k\in \N},\ n\in\N,
\]
each $x := \left(x_k\right)_{k \in \N}\in c(\N)$ has the Schauder expansion
\begin{equation*}%\label{Sec}
x = \sum_{k=0}^\infty c_k(x) e_k
\end{equation*}
with the \textit{coordinates}
\begin{equation*}%\label{Sec}
c_0(x)=l(x)\quad \text{and}\quad c_k(x)=x_k-l(x),\ k\in \N.
\end{equation*}

See, e.g., \cite{Markin2027EFA,Markin2026EOT,MarkSogh2021}.

%%%%%%%%%%%%%%%%%%%%%%%%%%%%%%%%%%%%%%%%%%%%%
\subsection{Spectrum}\

The spectrum $\sigma(A)$ of a closed linear operator $A$ in a complex Banach space $X$ is the union of the following pairwise disjoint sets:
\begin{equation*}
\begin{split}
& \sigma_p(A):=\left\{\lambda\in \C \,\middle|\,A-\lambda I\ \text{is \textit{not injective}, i.e., $\lambda$ is an \textit{eigenvalue} of $A$} \right\},\\
& \sigma_c(A):=\left\{\lambda\in \C \,\middle|\,A-\lambda I\ \text{is \textit{injective},
\textit{not surjective}, and $\overline{R(A-\lambda I)}=X$} \right\},\\
& \sigma_r(A):=\left\{\lambda\in \C \,\middle|\,A-\lambda I\ \text{is \textit{injective} and $\overline{R(A-\lambda I)}\neq X$} \right\}
\end{split}
\end{equation*}
($R(\cdot)$ is the \textit{range} of an operator and $\overline{\cdot}$ is the \textit{closure} of a set), called the \textit{point}, \textit{continuous} and \textit{residual spectrum} of $A$, respectively (see, e.g., \cite{Dun-SchI,Markin2026EOT}).

%%%%%%%%%%%%%%%%%%%%%%%%%%%%%%%%%%%%%%%%%%%%%
\subsection{Hypercyclicity and Linear Chaos}\

\begin{defn}[Hypercyclic and Chaotic Linear Operators]\ \\ \noindent
For a (bounded or unbounded) linear operator $A$ in a (real or complex) Banach space $X$, a nonzero vector 
\begin{equation*}
x\in C^\infty(A):=\bigcap_{n=0}^{\infty}D(A^n)
\end{equation*}
($D(\cdot)$ is the \textit{domain} of an operator, $A^0:=I$, $I$ is the \textit{identity operator} on $X$) is called \textit{hypercyclic} if its \textit{orbit} under $A$
\[
\orb(x,A):=\left\{A^nx\right\}_{n\in\Z_+}
\]
is dense in $X$, i.e.,
\[
\bar{\orb(x,A)}=X.
\]

Linear operators possessing hypercyclic vectors are said to be \textit{hypercyclic}.

If there exists an $N\in \N$ and a vector $x\in D(A^N)$ such that
\[
A^Nx = x,
\]
then this vector is called a \textit{periodic point} for the operator $A$ of period $N$. If $x\ne 0$, we say that $N$ is a \textit{period} for $A$.

A hypercyclic linear operator $A$ such that the set $\Per(A)$ of its periodic points is dense in $X$, i.e.,
\[
\bar{\Per(A)}=X,
\]
is called \textit{chaotic}.
\end{defn}

See \cite{Devaney,Godefroy-Shapiro1991,B-Ch-S2001}.

%\begin{samepage}
\begin{exmps}\label{exmps}\
\begin{enumerate}[label=\arabic*.]
    \item On the infinite-dimensional separable Banach space $X:=c_0$ or $X:=l_p$ ($1\le p<\infty$), the classical Rolewicz weighted backward shifts
    $$ X \ni x:=\left(x_k\right)_{k\in \N} \mapsto A_wx:= w\left(x_{k+1}\right)_{k\in \N} \in X, $$
    where $w\in \F$ with $|w|>1$, are \textit{chaotic} \cite{Rolewicz,Godefroy-Shapiro1991}.
   % whereas the multiplication operators $$ X \ni x:=\left(x_n\right)_{n\in \N} \mapsto B_wx:= w\left(x_n\right)_{n\in \N} \in X $$
 %%%%%%%%%%%%%%%%%%%%%%%%%%%%%%%%%%%%%%%%%%%%%%%%%%%%%%%%%%%%
   \item On the sequence space
   \[
X:=\left\{ \left(x_k\right)_{k\in \N}\in \F^\N\,\middle|\, 
\sum_{k=1}^\infty \left|\frac{x_{k}}{k}- \frac{x_{k+1}}{k+1}\right|<\infty\ \text{and}\ \lim_{k\to\infty}\frac{x_k}{k} = 0\right\},  
  \]
which is an infinite-dimensional separable Banach space relative to the norm
    $$X \ni x:=\left(x_k\right)_{k\in \N}\mapsto \|x\| :=\sum_{k=1}^\infty \left| \frac{x_{k}}{k}- \frac{x_{k+1}}{k+1} \right|, $$
 the weighted backward shifts
    $$ X \ni x:=\left(x_k\right)_{k\in \N} \mapsto A_wx:= w\left(x_{k+1}\right)_{k\in \N} \in X, $$
where $w\in \F$ with $|w|=1$, 
    are \textit{hypercyclic} but \textit{not} densely periodic,
    and hence, \textit{not}  chaotic \cite{arXiv:2209.04515}
(see also \cite{Grosse-Erdmann2000} and \cite[Exercise $4.1.3$]{Grosse-Erdmann-Manguillot}).
%%%%%%%%%%%%%%%%%%%
\item On an infinite-dimensional separable Banach space $(X,\|\cdot\|)$, 
the identity operator $I$ is densely periodic but \textit{not}
hypercyclic, and hence, \textit{not}  chaotic.
\end{enumerate}
\end{exmps}
%\end{samepage}

%\begin{samepage}
\begin{rems}\label{HCrems}\
\begin{itemize}
\item In the prior definition of hypercyclicity, the underlying space is necessarily
\textit{infinite-dimensional} and \textit{separable} (see, e.g., \cite{Grosse-Erdmann-Manguillot}).
%%%%%%%%
\item For a hypercyclic linear operator $A$, the set $HC(A)$ of its hypercyclic vectors is necessarily dense in $X$, and hence, the more so, is the subspace $C^\infty(A)\supseteq HC(A)$. In particular, this implies that, for any $n\in \N$,
the operator $A^n$ needs to be \textit{densely defined}, i.e.,
\[
\bar{D\left(A^n\right)}=X.
\]
%%%%%%%%%%%%%%%%%
\item Observe that
\[
\Per(A)=\bigcup_{N=1}^\infty \Per_N(A),
\]
where 
\[
\Per_N(A)=\ker(A^N-I),\ N\in \N
\]
is the \textit{subspace} of $N$-periodic points of $A$.
%%%%%%%%
\item As immediately follows from the inclusions
\begin{equation*}
HC(A^n)\subseteq HC(A),\ \Per(A^n)\subseteq \Per(A), n\in \N,
\end{equation*}
if, for a linear operator $A$ in an infinite-dimensional separable Banach space $X$ and some $n\ge 2$, the operator $A^n$ is hypercyclic/chaotic, then $A$ is also hypercyclic/chaotic, respectively.
\end{itemize} 
\end{rems}
%\end{samepage}

Prior to \cite{B-Ch-S2001,deL-E-G-E2003}, the notions of linear hypercyclicity and chaos had been studied exclusively for \textit{continuous} linear operators on Fr\'echet spaces, in particular for \textit{bounded} linear operators on Banach spaces (for a comprehensive survey, see \cite{Bayart-Matheron,Grosse-Erdmann-Manguillot}).

The following extension of \textit{Kitai's criterion} for bounded linear operators (see \cite{Kitai1982,Gethner-Shapiro1987}) is a useful shortcut for establishing hypercyclicity for (bounded or unbounded) linear operators without explicitly furnishing a hypercyclic vector as in \cite{Rolewicz}.

\begin{thm}[Sufficient Condition for Hypercyclicity {\cite[Theorem $2.1$]{B-Ch-S2001}}]\label{SCH}\ \\
Let $X$ be a (real or complex) infinite-dimensional separable Banach space and A be a densely defined linear operator in X such that each power $A^n$, $n \in \N$, is a closed operator. If there exists a set
\[
Y\subseteq C^\infty(A):=\bigcap_{n=1}^\infty D(A^n)
\]
dense in $X$ and a mapping $B:Y\to Y$ such that
\begin{enumerate}
\item $\forall\, x\in Y:\ ABx=x$ and
%%%%%%%%%%%%%%%%
\item $\forall\, x\in  Y:\ A^n x, B^n x \to 0,\ n \to \infty,$ 
\end{enumerate}
then the operator $A$ is hypercyclic.
\end{thm}

The subsequent newly established sufficient condition for linear chaos \cite{arXiv:2106.14872}, obtained via strengthening the second hypothesis of the prior statement, serves as a shortcut for establishing chaoticity for bounded or unbounded linear operators without explicitly furnishing a hypercyclic vector and a dense set of periodic points and is fundamental for our discourse.

\begin{thm}[Sufficient Condition for Linear Chaos {\cite[Theorem $3.2$]{arXiv:2106.14872}}]\label{SCC}\ \\
Let $(X,\|\cdot\|)$ be a  (real or complex) infinite-dimensional separable Banach space and $A$ be a densely defined linear operator in $X$ such that each power $A^{n}$, $n\in\N$, is a closed operator. If there exists a set
\[
Y\subseteq C^\infty(A):=\bigcap_{n=1}^\infty D(A^n)
\]
dense in $X$ and a mapping $B:Y\to Y$ such that
\begin{enumerate}
\item $\forall\, x\in Y:\ ABx=x$ and
%%%%%%%%%%%%%%%%
\item $\forall\, x\in Y\  \exists\, \alpha=\alpha(x)\in (0,1), c=c(x,\alpha)>0\ \forall\, n\in \N:$
\begin{equation*}
\max\left(\|A^nx\|,\|B^nx\|\right)\le c\alpha^n,
\end{equation*}
or equivalently,
\begin{equation}\label{(2(b))}
\forall\, x\in Y:\ \max\left(r(A,x),r(B,x)\right)<1,
\end{equation}
where 
\[
r(A,x):=\limsup_{n\to \infty}{\|A^nx\|}^{1/n}\quad \text{and}\quad
r(B,x):=\limsup_{n\to \infty}{\|B^nx\|}^{1/n},
\]
\end{enumerate}
then the operator $A$ is chaotic.
\end{thm}

For applications, see, e.g., \cite{arXiv:2106.09682,arXiv:2209.04515}.

\begin{rem}
The hypercyclic but not chaotic weighted backward shifts from the second of Examples \ref{exmps}, being subject to the \textit{Sufficient Condition for Linear Hypercyclicity} (Theorem \ref{SCH}), are not subject to the \textit{Sufficient Condition for Linear Chaos} (Theorem \ref{SCC}) \cite{arXiv:2209.04515}.
\end{rem}

We also need the following statements.

\begin{cor}[Chaoticity of Powers {\cite[Corollary $4.3$]{arXiv:2106.14872}}]\label{CP}\ \\
For a chaotic linear operator $A$ in a  (real or complex) infinite-dimensional separable Banach space subject to the \textit{Sufficient Condition for Linear Chaos} (Theorem \ref{SCC}), each power $A^n$, $n\in \N$, is chaotic.
\end{cor}

\begin{thm}[Bourdon {\cite[Theorem $2.54$]{Grosse-Erdmann-Manguillot}}]\label{BT}\ \\
For a bounded linear hypercyclic operator $A$ on an infinite-dimensional separable Banach space $X$ and a nonzero polynomial $p(\lambda):=\sum_{k=0}^n c_k\lambda^k$,
$n\in \Z_+$, $c_k\in \F$, $k=0,\dots,n$, the range $R(p(A))$ of the operator 
$p(A):=\sum_{k=0}^n c_kA^k$ is dense in $X$, i.e.,
\[
\bar{R(p(A))}=X.
\]
\end{thm}

\begin{rem}\label{remB}
Consistently with necessary conditions for hypercyclicity \cite[Proposition $4.1$]{arXiv:2106.14872}, the latter implies that, for a bounded linear hypercyclic operator $A$ on an infinite-dimensional separable Banach space $X$ and an arbitrary $\lambda\in \F$, with $p(\mu):=\mu-\lambda$, $\mu\in \F$, the range $R(A-\lambda I)$ of the operator $A-\lambda I$ is \textit{dense} in $X$, i.e.,
\[
\bar{R(A-\lambda I)}=X.
\]
\end{rem}

%%%%%%%%%%%%%%%%%%%%%%%%%%%%%%%%%%%%%%%%%%%%
\section{Linear Chaos in $c_0$}

%added
As is noted in the Introduction (Section \ref{intro}) and in Examples \ref{exmps}, the bounded weighted backward shifts in $c_0$ were introduced and studied in  \cite{Rolewicz,Godefroy-Shapiro1991}. We reestablish this result in the following statement via the \textit{Sufficient Condition for Linear Chaos} (Theorem \ref{SCC}) \cite[Theorem $3.2$]{arXiv:2106.14872} and further prove the chaoticity of all natural powers of these operators as well as analyze their spectral structure.

\begin{thm}[Bounded Linear Chaos on $c_0$]\label{BLCc0}\ \\
For an arbitrary $w\in \F$ with $|w|>1$, the bounded linear weighted backward shift operator 
\[
c_0\ni x:=\left(x_k\right)_{k\in \N}\mapsto A_wx:=w\left(x_{k+1}\right)_{k\in \N}\in c_0
\]
on the space $c_0$ is chaotic as well as its every power $A_w^n$, $n\in \N$, and, provided the underlying space is complex (i.e., $\F=\C$),
\[
\sigma\left(A_w\right)=\left\{ \lambda\in \C \,\middle|\, |\lambda|\le |w| \right\}
\]
with
\[
\sigma_p\left(A_w\right)=\left\{ \lambda\in \C \,\middle|\, |\lambda|<|w| \right\}\quad \text{and}\quad
\sigma_c\left(A_w\right)=\left\{ \lambda\in \C \,\middle|\, |\lambda|=|w| \right\}.
\]
\end{thm}

\begin{proof}
%Here, we provide a concise proof based on the \textit{Sufficient Condition for Linear Chaos} (Theorem \ref{SCC}) 

Let $w\in \F$ with $|w|>1$ be arbitrary and, for simplicity of notation, let $A:=A_w$.

Consider the subspace
\[
Y:=c_{00}:=\left\{ x:= (x_k)_{k \in \N}\in \F^\N \,\middle|\, \exists\, N\in \N\ \forall\,k\ge N:\ x_k=0 \right\},
\]
which is \textit{dense} in $c_0$ (see, e.g., \cite{Markin2026EOT,Markin2027EFA}), and the mapping $B:Y\to Y$, which is the restriction to $Y$ of the following \textit{bounded linear operator}
on $c_0$:
\[
c_0\ni x:=(x_k)_{k\in \N}\mapsto Bx:=w^{-1}\left(x_{k-1}\right)_{k\in \N}\in c_0 \quad (x_0:=0),
\]
(the right inverse of $A$) for which, as is easily seen,
\begin{equation}\label{BBS}
\|B\|={|w|}^{-1}<1
\end{equation}
(here and henceforth, whenever appropriate, the notation $\|\cdot\|$ is used for the \textit{operator norm}) (see, e.g., \cite{Markin2026EOT}) and
\begin{equation}\label{RI1}
ABx=x,\ x\in Y.
\end{equation}

Let us show that 
\[
\forall\, x\in Y\  \exists\, \alpha=\alpha(x)\in (0,1), c=c(x,\alpha)>0\ \forall\, n\in \N:\ \max\left(\|A^nx\|_\infty,\|B^nx\|_\infty\right)\le c\alpha^n.
\]

Let $x:= (x_k)_{k\in \N} \in Y$ be arbitrary. Then
\[
\exists\, N\in \N\ \forall\,k\ge N:\ x_k=0,
\]
and hence,
\[
\forall\,n\ge N:\ A^nx = 0,
\]
which implies that
\[
\forall\, \alpha\in (0,1)\ \exists\, c=c(x,\alpha)>0\ \forall\, n\in \N:\ 
\|A^nx\|_\infty\le c\alpha^n.
\]

By the submultiplicativity of the operator norm, in view of \eqref{BBS}, we also have: 
\begin{equation*}
\| B^n x\|_\infty \leq \| B^n \| \| x \|_\infty \leq \| B \|^n \| x \|_\infty = |w|^{-n} \| x \|_\infty.
\end{equation*}	

By the \textit{Sufficient Condition for Linear Chaos} (Theorem \ref{SCC}) and the \textit{Chaoticity of Powers Corollary} (Corollary \ref{CP}), we conclude that the operator $A$ is \textit{chaotic} as well as every power $A^n$, $n\in \N$. 

Provided the underlying space is complex, the spectral part of the statement
immediately follows from the fact that
\[
A=wL,
\]
where 
\[
c_0\ni x:=\left(x_k\right)_{k\in \N}\mapsto Lx:=\left(x_{k+1}\right)_{k\in \N}\in c_0
\]
is the backward shift operator on $c_0$, for which
\[
\sigma\left(L\right)=\left\{ \lambda\in \C \,\middle|\, |\lambda|\le 1 \right\}
\]
with
\[
\sigma_p\left(L\right)=\left\{ \lambda\in \C \,\middle|\, |\lambda|<1 \right\}\quad \text{and}\quad
\sigma_c\left(L\right)=\left\{ \lambda\in \C \,\middle|\, |\lambda|=1 \right\}
\]
(see, e.g., \cite{Dun-SchI,Markin2026EOT}).
\end{proof}

As is noted in the Introduction (Section \ref{intro}), the unbounded weighted backward shifts in $c_0$ (and $l_p$ ($1\le p<\infty$)) were introduced and studied in \cite{arXiv:1811.06640}. We reestablish this result in the subsequent theorem via the \textit{Sufficient Condition for Linear Chaos} (Theorem \ref{SCC}) \cite[Theorem $3.2$]{arXiv:2106.14872}, prove the chaoticity of all their natural powers, and replicate the analysis of their spectral structure provided in \cite{arXiv:1811.06640}. To this end, we need the following lemma.

\begin{lem}\label{lem}
Let $w\in \F$ and $|w|>1$. Then, for the weighted backward shift operator
\[
A_wx:=\left(w^kx_{k+1}\right)_{k\in \N}
\]
in the space $c_0$ with maximal domain
\[
D(A_w):=\left\{ x:=\left(x_k\right)_{k\in \N} \in c_0 \,\middle|\, \left(w^kx_{k+1}\right)_{k\in \N}\in c_0 \right\},
\]
each power 
\[
A_w^{n}x= \left( \left[ \prod_{j=k}^{k+n-1} w^j \right]x_{k+n} \right)_{k \in \N},
\]
$n\in\N$, with domain
\[
D(A_w^{n})= \left\{ x := \left(x_k \right)_{k \in \N} \in c_0 \,\middle|\, \left( \left[ \prod_{j=k}^{k+n-1} w^j \right]x_{k+n} \right)_{k \in \N} \in c_0 \right\}
\]
is a densely defined unbounded closed linear operator and the subspace
\[
C^\infty(A_w):=\bigcap_{n=1}^\infty D(A_w^n)
\]
of infinite differentiable relative to the operator $A_w$ vectors is dense in $c_0$.
\end{lem}

\begin{proof}
Let $w\in \F$ with $|w|>1$ be arbitrary and, for simplicity of notation, let $A:=A_w$.

Since
\[
A^2x= \left( w^{k} w^{k+1}x_{k+2} \right)_{k \in \N}
\]
with domain
\begin{align*}
D(A^{2})
&= \left\{ x := \left(x_k \right)_{k \in \N} \in D(A) \,\middle|\, Ax \in D(A) \right\}
\\
&= \left\{ x := \left(x_k \right)_{k \in \N} \in c_0 \,\middle|\,  \left( w^{k} w^{k+1}x_{k+2} \right)_{k \in \N} \in c_0 \right\}
\end{align*}
and 
\[
A^3x= \left( w^{k} w^{k+1}w^{k+2}x_{k+3} \right)_{k \in \N}
\]
with domain
\begin{align*}
D(A^{3})
&= \left\{ x := \left(x_k \right)_{k \in \N} \in D(A^2) \,\middle|\, A^2x \in D(A) \right\}
\\
&= \left\{ x := \left(x_k \right)_{k \in \N} \in c_0 \,\middle|\, \left( w^{k} w^{k+1}w^{k+2}x_{k+3} \right)_{k \in \N} \in c_0 \right\}
\end{align*}
we infer inductively that, for each $n \in \N$
\[
A^nx =  \left( \left[ \prod_{j=k}^{k+n-1} w^j \right]x_{k+n} \right)_{k \in \N},  
\]
with domain
\[
D(A^{n})= \left\{ x := \left(x_k \right)_{k \in \N} \in c_0 \,\middle|\, \left( \left[ \prod_{j=k}^{k+n-1} w^j \right]x_{k+n} \right)_{k \in \N} \in c_0 \right\}
\]

We have:
\[
D(A^{n+1}) \subseteq D(A^n),\ n \in \N.
\]
		
Since the subspace $c_{00}$ is \textit{dense} in $c_0$ and 
\[
c_{00} \subseteq D(A^n),\ n \in \N,
\]
then each power $A^n$ ($n\in \N$) is densely defined and furthermore
		\[
		C^\infty(A) := \bigcap_{n=1}^\infty D(A^n)
		\] 
is also \textit{dense} in $c_0$.

Let $n\in \N$ and $e_m:=\left(\delta_{mk} \right)_{k\in \N}$, $m\in\N$, with 
$\|e_m\|_\infty=1$, $m\in \N$. Then, in view of $|w|>1$,
\begin{align*}
\forall\, m\in \N:\ \|A^ne_{m+n}\|& = \left\| \left( \left[ \prod_{j=k}^{k+n-1} w^j \right] \delta_{(m+n)(k+n)} \right)_{k\in\N}  \right\|_\infty = \prod_{j=m}^{m+n-1} |w|^j\\
&= |w|^{ \sum_{j=m}^{m+n-1}j }= |w|^{\frac{n(2m+n-1)}{2}} \to \infty,\ m \to \infty,
\end{align*}
the linear operator $A^n$ is \textit{unbounded}.

Let $n\in \N$ and a sequence $\left(x^{(m)}:=\left(x_k^{(m)} \right)_{k\in \N} \right)_{m\in \N}$ in $D(A^n)$ be such that
		\[
		x^{(m)} \to x:=\left(x_k \right)_{k\in \N}\in c_0,\ m \to \infty,
		\]
		and
		\[
		A^nx^{(m)} = \left[ \prod_{j=k}^{k+n-1} w^j \right]x_{k+n}^{(m)} \to y:=\left(y_k \right)_{k\in \N}\in c_0,\ m \to \infty.
		\]

		Then, for each $k \in \N$ (see, e.g., \cite{Markin2027EFA,Markin2026EOT,MarkSogh2021}),
		\[
		x_k^{(m)} \to x_k,\ m \to \infty,
		\]
		and 
		\[
		\left[ \prod_{j=k}^{k+n-1} w^j \right]x_{k+n}^{(m)} \to y_k, \ m \to \infty.
		\]
		Whence we infer that, for each $k\in \N$,
\[
\left[ \prod_{j=k}^{k+n-1} w^j \right]x_{k+n} = y_k,
\]
which means that
\[
\left( \left[ \prod_{j=k}^{k+n-1} w^j \right]x_{k+n} \right)_{k \in \N}=y\in c_0.
\]
		
		Therefore, $x \in D(A^n)$ and $y = A^nx$, which, by the \textit{Sequential Characterization of Closed Linear Operators} (see, e.g., \cite{Markin2027EFA,Markin2026EOT}), implies the operator $A^n$ is closed.
\end{proof}

\begin{thm}[Unbounded Linear Chaos in $c_0$]\label{ULCc0}\ \\
For an arbitrary $w\in \F$ with $|w|>1$, the unbounded linear weighted backward shift operator
\[
A_wx:=\left(w^kx_{k+1}\right)_{k\in \N}
\]
in the space $c_0$ with maximal domain
\[
D(A_w):=\left\{ x:=\left(x_k\right)_{k\in \N} \in c_0 \,\middle|\, \left(w^kx_{k+1}\right)_{k\in \N}\in c_0 \right\}
\]
is chaotic as well as its every power $A_w^n$, $n\in \N$.

Furthermore, each $\lambda \in \F$ is an eigenvalue for $A_w$ of geometric multiplicity $1$, i.e.,
\[
\dim\ker(A_w-\lambda I)=1.
\]
In particular, provided the underlying space is complex,
\[
\sigma_p\left(A_w\right)=\C.
\]
\end{thm}

\begin{proof}
%Here, we also provide a concise proof based on the \textit{Sufficient Condition for Linear Chaos} (Theorem \ref{SCC}) (cf. the original proof \cite{arXiv:1811.06640}).

Let $w\in \F$ with $|w|>1$ be arbitrary and, for simplicity of notation, let $A:=A_w$.

Consider the \textit{dense} subspace
\[
Y:=c_{00}
\]
of $c_0$ and the mapping $B:Y\to Y$, which is the restriction to $Y$ of the 
\textit{bounded linear operator} on $c_0$
\[
c_0\ni x:=(x_k)_{k\in \N}\mapsto Bx:=\left(w^{-(k-1)}x_{k-1}\right)_{k\in \N}\in c_0\quad (x_0:=0),
\]
(the right inverse of $A$) for which
\begin{equation*}%\label{RI2}
ABx=x,\ x\in c_0.
\end{equation*}

In particular, the latter holds for any $x\in Y$.

With
\[
c_0\ni x:=(x_k)_{k\in \N}\mapsto B^2x= \left( w^{-(k-1)} w^{-(k-2)} x_{k-2} \right)_{k \in \N}\quad (x_{k-2}:=0,\ k=1,2)
\]
and 
\[
c_0\ni x:=(x_k)_{k\in \N}\mapsto B^3x= \left( w^{-(k-1)} w^{-(k-2)} w^{-(k-3)} x_{k-3} \right)_{k \in \N}
\]
($x_{k-3}:=0,\ k=1,2,3$), we infer inductively that, for any $n\in \N$,
\[
c_0\ni x:=(x_k)_{k\in \N}\mapsto B^nx = \left(\left[ \prod_{j=1}^{n} w^{-(k-j)} \right] x_{k-n} \right)_{k \in \N}\quad (x_{k-n}:=0,\ k=1,\dots,n),
\]
or equivalently, in view of
\[
\prod_{j=1}^{n} w^{-(k-j)}=w^{-\sum_{j=1}^n (k-j)}=w^{-nk + \frac{n(n+1)}{2}},
\]
we have:
\[
c_0\ni (x_k)_{k\in \N}\mapsto B^nx = \left( w^{-nk + \frac{n(n+1)}{2}}x_{k-n} \right)_{k \in \N}\quad (x_k:=0,\ k=1,\dots,n).
\]

Let us show that
\begin{equation*}
\|B^n\| ={|w|}^{-\frac{n(n+1)}{2}},\ n\in \N.
\end{equation*}

Indeed, for any $n\in \N$ and $x:=\left(x_k\right)_{k\in \N}\in c_0$, in view of $|w|>1$,
\begin{align*}
\|B^nx\|_\infty&=\sup_{k\in \N}\left|w^{-nk + \frac{n(n+1)}{2}}x_{k-n}\right|
=\sup_{k\ge n+1}\left|w^{-nk + \frac{n(n+1)}{2}}x_{k-n}\right|
\\
&\le \sup_{k\ge n+1}{\left|w\right|}^{-nk + \frac{n(n+1)}{2}}
\sup_{k\ge n+1}\left|x_{k-n}\right|
={\left|w\right|}^{-n(n+1) + \frac{n(n+1)}{2}}\|x\|_\infty
\\
&={\left|w\right|}^{-\frac{n(n+1)}{2}}\|x\|_\infty,
\end{align*}
and hence,
\[
\|B^n\| \le {\left|w\right|}^{-\frac{n(n+1)}{2}}.
\]

Further, for any $n\in \N$, since, for $e_{1}:=\left(\delta_{1k}\right)_{k\in \N}$ with $\|e_{1}\|_\infty=1$,
\[
\|B^ne_1\|_\infty=\sup_{k\ge n+1}\left|w^{-nk + \frac{n(n+1)}{2}}\delta_{1(k-n)}\right|
={\left|w\right|}^{-n(n+1) + \frac{n(n+1)}{2}}={\left|w\right|}^{-\frac{n(n+1)}{2}},
\]
we infer that
\[
\|B^n\| ={\left|w\right|}^{-\frac{n(n+1)}{2}}.
\]

Thus,
\begin{equation}\label{UBS}
\lim_{n\to \infty}{\|B^n\|}^{1/n}=\lim_{n\to \infty}{\left|w\right|}^{-\frac{n+1}{2}}=0,
\end{equation}
i.e., the operator $B$ is \textit{quasinilpotent} (cf. \cite{arXiv:1811.06640}), which implies that
\begin{equation*}
\begin{aligned}
\forall\, x\in c_0:\ & \limsup_{n\to \infty}\|B^nx\|_\infty^{1/n}
\le \limsup_{n\to \infty}{\left(\|B^n\|\|x\|_\infty\right)}^{1/n} \\
&=\lim_{n\to \infty}{\|B^n\|}^{1/n} \lim_{n\to \infty}\|x\|_\infty^{1/n}=0<1.
\end{aligned}
\end{equation*}

%\[
%\forall\, x\in Y, \alpha\in (0,1)\ \exists\, N\in \N\ \forall\, n\ge N:\ 
%{\|B^nx\|_\infty}^{1/n}\le {\|B^n\|}^{1/n}{\|x\|}^{1/n}_\infty\le c\alpha^n.
%\]

%\[
%\forall\, x\in Y, \alpha\in (0,1)\ \exists\, c=c(x,\alpha)>0\ \forall\, n\in \N:\ 
%\|B^nx\|_\infty\le c\alpha^n.
%\]

In particular, the latter holds for any $x\in Y$.

Let $x:= (x_k)_{k\in \N} \in Y$ be arbitrary. Then
\[
\exists\, N\in \N\ \forall\,k\ge N:\ x_k=0,
\]
and hence,
\[
\forall\,n\ge N:\ A^nx = 0,
\]
which implies that
\[
\limsup_{n\to \infty}\|A^nx\|_\infty^{1/n}=\lim_{n\to \infty}\|A^nx\|_\infty^{1/n}=0.
\]

%\[
%\forall\, \alpha\in (0,1)\ \exists\, c=c(x,\alpha)>0\ \forall\, n\in \N:\ %\|A^nx\|_\infty\le c\alpha^n.
%\]

By Lemma \ref{lem}, the \textit{Sufficient Condition for Linear Chaos} (Theorem \ref{SCC}), and the \textit{Chaoticity of Powers Corollary} (Corollary \ref{CP}), we conclude that the operator $A$ is \textit{chaotic} as well as every power $A^n$, $n\in \N$. 

Here, we reproduce the proof of the spectral part of the statement furnished in \cite{arXiv:1811.06640}.
%\cite[Theorem $3.1$]{arXiv:1811.06640}.

For arbitrary $\lambda \in \F$ ($\F:=\R$ or $\F:=\C$) and $x:=(x_k)_{\N} \in D(A)$, 
\begin{equation}\label{ev}
Ax=\lambda x,
\end{equation}
or equivalently,
\[
(w^k x_{k+1})_{k\in \N}=\lambda(x_k)_{k\in \N},
\]
implies that
\[
w^k x_{k+1}=\lambda x_k,\ k\in \N
\]

Whence, we recursively infer that
\[
x_k=\left[\prod_{j=1}^{k-1}\frac{\lambda}{w^{k-j}}\right]x_1
=\dfrac{\lambda^{k-1}}{w^{\sum_{j=1}^{k-1}(k-j)}}x_1
=\dfrac{\lambda^{k-1}}{w^{\frac{k(k-1)}{2}}}x_1
=\left(\dfrac{\lambda}{w^{\frac{k}{2}}}\right)^{k-1}x_1,\
k\in \N,
\]
with $0^0:=1$, as usual.
%where for $\lambda=0$, 

Considering that $|w|>1$, for all sufficiently large $k\in \N$, we have:
\[
\left|\dfrac{\lambda}{w^{\frac{k}{2}}}\right|^{k-1}
=\left(\dfrac{|\lambda|}{|w|^{\frac{k}{2}}}\right)^{k-1}
\le \left(\dfrac{1}{2}\right)^{k-1},
\]
which implies that
\[
y:=(y_k)_{k\in \N}:=\left(\left(\dfrac{\lambda}{w^{\frac{k}{2}}}\right)^{k-1}\right)_{k\in \N}\in c_0.
\]

Further, since
\[
w^ky_{k+1}=w^k\dfrac{\lambda^{k}}{w^{\sum_{j=1}^{k}(k+1-j)}}
=\dfrac{\lambda^{k}}{w^{\sum_{j=2}^{k}(k+1-j)}}=\dfrac{\lambda^{k}}{w^{\frac{(k-1)k}{2}}}=\left(\dfrac{\lambda}{w^{\frac{k-1}{2}}}\right)^{k},\ k\in \N,
\]
we similarly conclude that
\[
(w^ky_{k+1})_{k\in \N}\in c_0,
\]
and hence,
\[
y\in D(A)\setminus \{0\}.
\]

Thus, we have shown that, for any $\lambda \in \F$, all solutions of equation \eqref{ev} are of the form
\[
x:=(x_k)_{\N}=cy\in D(A),
\]
where $c\in \F$ is arbitrary. They form the \textit{one-dimensional} subspace of $c_0$ spanned by the sequence $y$, which completes the proof.
\end{proof}

\begin{rem}\label{NZ}
Theorem \ref{BLCc0}, Lemma \ref{lem}, and Theorem \ref{ULCc0} naturally extend from $c_0(\N)$ to $c_0(\Z_+)$ for the bounded weighted backward shifts:
\[
c_0(\Z_+)\ni x:=\left(x_k\right)_{k\in \Z_+}\mapsto A_wx:=w\left(x_{k+1}\right)_{k\in \Z_+}\in c_0(\Z_+) \quad (|w|>1)
\]
and the unbounded weighted backward shifts:
\[
A_wx:=\left(w^kx_{k+1}\right)_{k\in \Z_+}\quad (|w|>1)
\]
with maximal domain
\[
D(A_w):=\left\{ x:=\left(x_k\right)_{k\in \Z_+} \in c_0(\Z_+) \,\middle|\, \left(w^kx_{k+1}\right)_{k\in \Z_+}\in c_0(\Z_+) \right\}
\]
and the powers
\[
A_w^{n}x= \left( \left[ \prod_{j=k}^{k+n-1} w^j \right]x_{k+n} \right)_{k \in \Z_+},\ n\in \N,
\]
defined on
\[
D(A_w^{n})= \left\{ x := \left(x_k \right)_{k \in \Z_+} \in c_0(\Z_+) \,\middle|\, \left( \left[ \prod_{j=k}^{k+n-1} w^j \right]x_{k+n} \right)_{k \in \Z_+} \in c_0(\Z_+) \right\}
\]
(see the proof of Lemma \ref{lem}).

In the former case, the bounded right inverse of $A_w$ is
\[
c_0(\Z_+)\ni x:=(x_k)_{k\in \Z_+}\mapsto B_w x:=w^{-1}\left(x_{k-1}\right)_{k\in \Z_+}\in c_0(\Z_+)\quad (x_{-1}:=0),
\]
for which $\|B\|={|w|}^{-1}<1$, and, in the latter case, the bounded right inverse of $A_w$ is
\[
c_0(\Z_+)\ni x:=(x_k)_{k\in \Z_+}\mapsto B_wx:=\left(w^{-(k-1)}x_{k-1}\right)_{k\in \Z_+}\in c_0(\Z_+)\quad (x_{-1}:=0),
\]
with
\begin{align*}
c_0(\Z_+)\ni x:=(x_k)_{k\in \Z_+}\mapsto B_w^nx &= \left(\left[ \prod_{j=1}^{n} w^{-(k-j)} \right] x_{k-n} \right)_{k \in \Z_+}\\
&
=\left( w^{-nk + \frac{n(n+1)}{2}}x_{k-n} \right)_{k \in \Z_+},\ n\in \N,
\end{align*}
($x_{k-n}:=0$, $k=0,1,\dots,n-1$), for which
\[
\|B_w^n\|={\left|w\right|}^{-n^2 + \frac{n(n+1)}{2}}={\left|w\right|}^{-\frac{(n-1)n}{2}},
\]
and hence,
\begin{equation*}
\lim_{n\to \infty}{\|B_w^n\|}^{1/n}=\lim_{n\to \infty}{\left|w\right|}^{-\frac{n-1}{2}}=0,
\end{equation*}
i.e., $B_w$ is \textit{quasinilpotent} (cf. the proof of Theorem \ref{ULCc0}).
\end{rem}

%%%%%%%%%%%%%%%%%%%%%%%%%%%%%%%%%%%%%%%%%%%%%
\section{Weighted Backward Shifts in $c$}\label{WBSc}

The answer to the natural question of whether one can obtain linear chaos in the space $c(\N)$ of convergent sequences by merely extending the foregoing chaotic weighted backward shifts from the space $c_0(\N)$ is given in the negative by the subsequent statements.  

\begin{prop}[Bounded Weighted Backward Shifts on $c$]\ \\
For an arbitrary $w\in \F$ with $|w|>1$, the bounded linear weighted backward shift operator 
\begin{equation*}%\label{bwbsc}
c\ni x:=\left(x_k\right)_{k\in \N}\mapsto A_wx:=w\left(x_{k+1}\right)_{k\in \N}\in c
\end{equation*}
on the space $c$ is not hypercyclic.
\end{prop}

\begin{proof}
Let $w\in \F$ with $|w|>1$ be arbitrary and, for simplicity of notation, let $A:=A_w$.

It is obvious that the operator $A$ is well defined on $c$ and also is linear and bounded with
\[
\|A\|=|w|.
\]

Since, for any $x:=\left(x_{k}\right)_{k \in \N}\in c$,
\begin{equation*}
(A - wI)x = w\left(x_{k+1}\right)_{k \in \N} - w\left(x_k\right)_{k \in \N}
= w\left(x_{k+1} - x_k\right)_{k \in \N}.
\end{equation*}
and
\[
\lim_{k\to \infty}w(x_{k+1} - x_k)=w\left(\lim_{k\to \infty}x_{k+1} - \lim_{k\to \infty}x_k\right)
=w\left(l(x) - l(x)\right)=0
\]
(see \eqref{lf}), we infer that
\[
R(A-wI)\subseteq c_0.
\]

Since $c_0$ is a closed proper subspace of $c$, it is nowhere dense in $c$
(see, e.g., \cite{Markin2027EFA,Markin2026EOT}) and, as follows from the prior inclusion,
so is $R(A-wI)$.

Hence,
\[
\overline{R(A_w-wI)}\neq c,
\]
which, by \textit{Bourdon's Theorem} (Theorem \ref{BT}) 
%with $p(\lambda):=\lambda -w$, $\lambda\in \F$ (see also \cite[Proposition $4.1$]{arXiv:2106.14872}), the latter 
implies that the operator $A$ is not hypercyclic (see Remark \ref{remB}).
\end{proof}

\begin{prop}[Unbounded Weighted Backward Shifts in $c$]\ \\
For an arbitrary $w\in \F$ with $|w|>1$, the unbounded linear weighted backward shift operator
\[
A_wx:=\left(w^kx_{k+1}\right)_{k\in \N}
\]
in the space $c$ with maximal domain
\[
D(A_w):=\left\{ x:=\left(x_k\right)_{k\in \N} \in c \,\middle|\, \left(w^kx_{k+1}\right)_{k\in \N}\in c\right\}
\]
is not hypercyclic.
\end{prop}

\begin{proof}
Let $w\in \F$ with $|w|>1$ be arbitrary and, for simplicity of notation, let $A:=A_w$.

As follows from the definition, for any $x:=\left(x_k\right)_{k\in \N}\in D(A)$,
\[
y:=\left(y_k:=w^kx_{k+1}\right)_{k\in \N}\in c
\]
and hence, in view of $|w|>1$,
\[
x_{k+1}=w^{-k}y_k\to 0,\ k\to \infty.
\]

Therefore,
\[
D(A)\subseteq c_0.
\]

Since $c_0$ is a closed proper subspace of $c$, it is nowhere dense in $c$
(see, e.g., \cite{Markin2027EFA,Markin2026EOT}) and, by the preceding inclusion, so is $D(A)$.

Hence,
\[
\overline{D(A)}\neq c,
\]
which immediately implies that the operator $A$ is not hypercyclic (see Remarks \ref{HCrems}).
\end{proof}

%%%%%%%%%%%%%%%%%%%%%%%%%%%%%%%%%%%%%%%%%%%%%
\section{Linear Chaos in $c$}\label{LCc}

With the hypercyclicity by extension compromised, here, we construct bounded and unbounded chaotic linear operators in $c(\N)$ based on the chaotic backward shifts in $c_0(\Z_+)$ via establishing a homeomorphic isomorphism between the two spaces (i.e., an isomorphism which is also a homeomorphism).

\begin{lem}[Homeomorphic Isomorphism]\label{lem1}\ \\ \noindent
The mapping
\[
c(\N) \ni x:= (x_k)_{k \in \N} \mapsto Jx:= (y_k)_{k \in \Z_+} \in c_0(\Z_+),
\]
assigning to each $x:= (x_k)_{k \in \N}\in c(\N)$  the sequence $\left(y_k\right)_{k\in \Z_+}\in c_0(\Z_+)$ of the coordinates of $x$ relative to the standard Schauder basis 
$\left\{e_n\right\}_{n\in \Z_+}$ for $c(\N)$, where 
\[
e_0 := (1,1,1,\dots)\quad \text{and}\quad e_n:=\left(\delta_{nk}\right)_{k\in \N},\ n\in\N,
\]
i.e.,
\[
y_0:=l(x)\quad \text{and}\quad y_k:=x_k-l(x),\ k\in \N,
\]
where $l$ is the limit functional, is a homeomorphic isomorphism of $c(\N)$ onto $c_0(\Z_+)$.
\end{lem}

\begin{proof}
In view of the uniqueness of the Schauder expansion, we infer that the mapping $J$ is \textit{linear} and further, since, for an $x\in c(\N)$,
\[
Jx = 0 \iff y_k = 0,\ k \in \Z_+ \iff x = \sum^\infty_{k=0}y_k e_k = 0 \in c(\N),
\]
$J$ is also \textit{injective} (see, e.g., \cite{Markin2027EFA,Markin2026EOT,MarkSogh2021}).

Further, for any $y:=(y_k)_{k \in \Z_+} \in c_0(\Z_+)$, let
\[
x:=\left( y_k+y_0\right)_{k\in \N}.
\]

Since
\[
\lim_{k\to \infty}y_k=0,
\]
we infer that
\[
\lim_{k\to \infty}x_k=\lim_{k\to \infty}(y_k+y_0)=y_0,
\]

%then $c_0 = l(x),~c_k = x_k - l(x), ~k \in \N.$  This in turn gives us that $x_k = c_k + c_0, k \in \N$.  Since $c_k \to 0, k \to \infty,$ we obtain $x_k = c_k + c_0 \to 0, k \to \infty. $  

Thus,
\[
x\in c(\N)\quad \text{and}\quad Jx=y,
\]
which implies that the mapping $J$ is also \textit{surjective}, and thus, is  \textit{bijective}.

Hence $J:c(\N)\to c_0(\Z_+)$ is an \textit{isomorphism} between the spaces $c(\N)$ and $c_0(\Z_+)$.

%The requirement that $J$ be a topological isomorphism is essential. \\

Since, for an arbitrary $x:=\left(x_k\right)_{k\in \N}\in c(\N)$,
\[
|l(x)|=\left|\lim_{k\to \infty}x_k\right|=\lim_{k\to \infty}\left|x_k\right|\le \sup_{k\in \N}|x_k|=:\|x\|_\infty,
\]
we also have:
\begin{equation*}
\| Jx \|_\infty := \sup_{k \in \Z_+} \left|y_k\right|
=\max\left[|l(x)|,\sup_{k \in \N} \left|x_k - l(x)\right|\right]
\leq 2 \| x \|_\infty.
\end{equation*}

Thus, the linear mapping $J$ is \textit{bounded} with $\| J \|\le 2$ (it can be easily shown that, in fact, $\| J \|=2$), and hence \textit{continuous}, which, by the \textit{Inverse Mapping Theorem} (see, e.g., \cite{Markin2027EFA,Markin2026EOT}), implies that so is its inverse $J^{-1}:c_0(\Z_+)\to c(\N)$:
\begin{equation}\label{JI}
c_0(\Z_+) \ni y:= (y_k)_{k \in \Z_+} \mapsto J^{-1}y:= (y_k+y_0)_{k \in \N} \in c(\N).
\end{equation}

Furthermore, for an arbitrary $y:=(y_k)_{k \in \Z_+} \in c_0(\Z_+)$,
\begin{equation*}
\left\| J^{-1}y \right\|_\infty := \sup_{k \in \N} \left|y_k+y_0\right|
\le \| y \|_\infty +|y_0| \le  2 \| y\|_\infty,
\end{equation*}
and hence,
\[
\left\| J^{-1}\right\|\le 2
\]
(it can be easily shown as well that actually $\left\| J^{-1}\right\|=2$).

Whence, we conclude that the mapping $J:c(\N)\to c_0(\Z_+)$ is both \textit{isomorphic} and \textit{homeomorphic}.
\end{proof}

\begin{thm}[Bounded Linear Chaos on $c$]\label{BLC}\ \\
For an arbitrary $w\in \F$ with $|w|>1$, the bounded linear operator 
\[
c\ni x:=\left(x_k\right)_{k\in \N}\mapsto \hat{A}_wx:=w\left(x_{k+1}+x_1-2l(x)\right)_{k\in \N}\in c
\]
on the space $c$ is chaotic as well as its every power 
\[
c\ni x:=\left(x_k\right)_{k\in \N}\mapsto \hat{A}_w^nx = w^n \left( x_{k+n} + x_n -2l(x) \right)_{k \in \N},\ n \in \N,  
\]
and, provided the underlying space is complex
\[
\sigma\left(\hat{A}_w\right)=\left\{ \lambda\in \C \,\middle|\, |\lambda|\le |w| \right\}
\]
with
\[
\sigma_p\left(\hat{A}_w\right)=\left\{ \lambda\in \C \,\middle|\, |\lambda|<|w| \right\}\quad \text{and}\quad
\sigma_c\left(\hat{A}_w\right)=\left\{ \lambda\in \C \,\middle|\, |\lambda|=|w| \right\}.
\]
\end{thm}

\begin{proof}
Let $w\in \F$ with $|w|>1$ be arbitrary and, for simplicity of notation, let 
$\hat{A}:=\hat{A}_w$ and $A:=A_w$.

On $c(\N)$, consider the linear operator $\hat{A}$ defined as follows:
\begin{equation}\label{sim}
\hat{A}:=J^{-1}AJ, 
\end{equation}
where
\[
c_0(\Z_+)\ni y:=\left(y_k\right)_{n\in \Z_+} \mapsto Ay:=w\left(y_{k+1}\right)_{k\in \Z_+}\in c_0(\Z_+)
\]
is the bounded weighted backward shift on $c_0(\Z_+)$, shown to be \textit{chaotic} 
along with its every power $A^n$, $n\in \N$, in Theorem \ref{BLCc0} (see Remark \ref{NZ}), and $J:c(\N)\to c_0(\Z_+)$ is the \textit{homeomorphic isomorphism} of $c(\N)$ onto $c_0(\Z_+)$ of Lemma \ref{lem1}, i.e., via the commutative diagram
\begin{equation*}
\begin{tikzcd}[sep=large]
c_0(\Z_+) \arrow[r, "A"] & c_0(\Z_+)  \\
c(\N)\arrow[u, "J"] \arrow[r, "\hat{A}"]& c(\N)\arrow[u, "J"].
\end{tikzcd}
\end{equation*}

Since, by \eqref{sim},
\[
{\hat{A}}^n:=J^{-1}A^nJ,\ n\ \in \N,
\]
where
\[
c_0(\Z_+)\ni y:=\left(y_k\right)_{n\in \Z_+} \mapsto A^ny:=w^n\left(y_{k+n}\right)_{k\in \Z_+}\in c_0(\Z_+)
\]
(see Remark \ref{NZ}), for any $x:=\left(x_k\right)_{k\in \N} \in c(\N)$,
\[
A^nJx=w^n\left(x_{k+n}-l(x)\right)_{k\in \Z_+}=:\left(y_k\right)_{k\in \Z_+},
\]
and hence, in view of \eqref{JI}, we have:
\[
{\hat{A}}^nx:=J^{-1}A^nJx=\left(y_k+y_0\right)_{k\in \N}
=w^n\left(x_{k+n}+x_n-2l(x)\right)_{k\in \N}\in c(\N).
\]

Observe that
\[
\lim_{k\to \infty}w^n\left(x_{k+n}+x_n-2l(x)\right)=w^n\left(l(x)+x_n-2l(x)\right)
=w^n\left(x_n-l(x)\right),\ n\in \N.
\]

Since, by Lemma \ref{lem1}, $J:c(\N)\to c_0(\Z_+)$ is a \textit{homeomorphic isomorphism}, the operator ${\hat{A}}^n$ ($n\in \N$) inherits \textit{linearity}, \textit{boundedness}, \textit{chaoticity}, and \textit{spectral structure} directly from its conjugate $A^n$ via $J$ (for a similar construct in the context of hypercyclicity, see, e.g., \cite[Theorem $2.4$]{B-Ch-S2001}). 

Therefore, the statement follows immediately from Theorem \ref{BLCc0}.
\end{proof}

\begin{thm}[Unbounded Linear Chaos in $c$]\label{ULC}\ \\
For an arbitrary $w\in \F$ with $|w|>1$, the linear operator 
\[
\hat{A}_wx:=\left(w^k(x_{k+1}-l(x))+x_1-l(x)\right)_{k\in \N}
\]
in the space $c$ with domain
\[
D(\hat{A}_w):=\left\{ x:=\left(x_k\right)_{k\in \N} \in c \,\middle|\, \left(w^k\left(x_{k+1}-l(x)\right)\right)_{k\in \N}\in c_0 \right\}
\]
is unbounded, closed, and chaotic as well as its every power 
\[
\hat{A}_w^nx=\left( \left[ \prod_{j=k}^{k+n-1} w^j \right]\left(x_{k+n}-l(x)\right)
+\left[ \prod_{j=0}^{n-1} w^j \right]\left(x_{n}-l(x)\right) \right)_{k \in \N},\ n\in \N,
\]
with domain
\[
D(\hat{A}_w^n)=\left\{ x:=\left(x_k\right)_{k\in \N} \in c \,\middle|\, 
\left( \left[ \prod_{j=k}^{k+n-1} w^j \right]\left(x_{k+n}-l(x)\right) \right)_{k \in \N}\in c_0 \right\}
\]

Furthermore, each $\lambda \in \F$ is an eigenvalue for $\hat{A}_w$ of geometric multiplicity $1$, i.e.,
\[
\dim\ker(\hat{A}_w-\lambda I)=1.
\]
\end{thm}

\begin{proof}
Let $w\in \F$ with $|w|>1$ be arbitrary and, for simplicity of notation, let 
$\hat{A}:=\hat{A}_w$ and $A:=A_w$.

In $c(\N)$, consider the linear operator $\hat{A}$ defined as follows:
\begin{equation}\label{sim1}
\hat{A}:=J^{-1}AJ, 
\end{equation}
where
\[
Ay:=\left(w^ky_{k+1}\right)_{k\in \Z_+}\quad (|w|>1)
\]
with maximal domain
\[
D(A):=\left\{ y:=\left(y_k\right)_{k\in \Z_+} \in c_0(\Z_+) \,\middle|\, \left(w^ky_{k+1}\right)_{k\in \Z_+}\in c_0(\Z_+) \right\}
\]
is the unbounded weighted backward shift  in $c_0(\Z_+)$, shown to be \textit{chaotic} 
along with its every power $A^n$, $n\in \N$, in Theorem \ref{ULCc0} (see Remark \ref{NZ}), and $J:c(\N)\to c_0(\Z_+)$ is the \textit{homeomorphic isomorphism} of $c(\N)$ onto $c_0(\Z_+)$ of Lemma \ref{lem1}, i.e., via the commutative diagram
\begin{equation*}
\begin{tikzcd}[sep=large]
D(A) \arrow[r, "A"] & c_0(\Z_+)  \\
D(\hat{A})\arrow[u, "J"] \arrow[r, "\hat{A}"]& c(\N)\arrow[u, "J"],
\end{tikzcd}
\end{equation*}
for which the domain is
\[
D(\hat{A}):=J^{-1}\left(D(A)\right).
\]

Since, by \eqref{sim1},
\[
{\hat{A}}^n:=J^{-1}A^nJ,\ n\ \in \N,
\]
where
\[
A^{n}y= \left( \left[ \prod_{j=k}^{k+n-1} w^j \right]y_{k+n} \right)_{k \in \Z_+}
\]
with domain
\[
D(A^{n})= \left\{ y := \left(y_k \right)_{k \in \Z_+} \in c_0(\Z_+) \,\middle|\, 
\left( \left[ \prod_{j=k}^{k+n-1} w^j \right]y_{k+n} \right)_{k \in \Z_+} \in c_0(\Z_+) \right\}
\]
(see Remark \ref{NZ}), we have:
\[
D({\hat{A}}^n)=\left\{ x:=\left(x_k\right)_{k\in \N} \in c(\N) \,\middle|\, 
\left( \left[ \prod_{j=k}^{k+n-1} w^j \right]\left(x_{k+n}-l(x)\right) \right)_{k \in \N}\in c_0(\N) \right\}
\]
and, considering that, for any $x:=\left(x_k\right)_{k\in \N} \in D({\hat{A}}^n)$,
\[
A^nJx=\left( \left[ \prod_{j=k}^{k+n-1} w^j \right]\left(x_{k+n}-l(x)\right) \right)_{k \in \Z_+}=:\left(y_k\right)_{k\in \Z_+},
\]
in view of \eqref{JI},
\begin{align*}
{\hat{A}}^nx&=J^{-1}A^nJx=\left(y_k+y_0\right)_{k\in \N}
\\
&=\left( \left[ \prod_{j=k}^{k+n-1} w^j \right]\left(x_{k+n}-l(x)\right)
+\left[ \prod_{j=0}^{n-1} w^j \right]\left(x_{n}-l(x)\right) \right)_{k \in \N}\in c(\N).
\end{align*}

Observe that
\begin{align*}
&\lim_{k\to \infty}\left( \left[ \prod_{j=k}^{k+n-1} w^j \right]\left(x_{k+n}-l(x)\right)
+\left[ \prod_{j=0}^{n-1} w^j \right]\left(x_{n}-l(x)\right) \right)\\
&=\left[ \prod_{j=0}^{n-1} w^j \right]\left(x_{n}-l(x)\right)
\end{align*}
since
\[
\left( \left[ \prod_{j=k}^{k+n-1} w^j \right]\left(x_{k+n}-l(x)\right) \right)_{k \in \N}\in c_0(\N)
\]
and 
\[
\left[ \prod_{j=0}^{n-1} w^j \right]\left(x_{n}-l(x)\right)\in \F
\]
is a constant independent of $k\in \N$.

In particular, for $n=1$, we have
\[
\hat{A}x:=\left(w^k(x_{k+1}-l(x))+x_1-l(x)\right)_{k\in \N}
\]
with domain
\[
D(\hat{A}):=\left\{ x:=\left(x_k\right)_{k\in \N} \in c \,\middle|\, \left(w^k\left(x_{k+1}-l(x)\right)\right)_{k\in \N}\in c_0 \right\}
\]
and
\[
\lim_{k\to \infty}\left(w^k(x_{k+1}-l(x))+x_1-l(x)\right)=x_1-l(x).
\]

Since, by Lemma \ref{lem1}, $J:c(\N)\to c_0(\Z_+)$ is a \textit{homeomorphic isomorphism}, the operator ${\hat{A}}^n$ ($n\in \N$) inherits \textit{linearity}, \textit{unboundedness}, \textit{closedness}, \textit{chaoticity}, and \textit{eigenvalues} along with their \textit{geometric multiplicities} directly from its conjugate $A^n$ via $J$. 

Therefore, the statement follows immediately from Theorem \ref{ULCc0}.
\end{proof}

%%%%%%%%%%%%%%%%Bibliography%%%%%%%%%%%%%%%%%%%%%%%
 
%%%%%%%%%%%%%%%%%%%%%%%%%%%%%%%%%%%%%%%%%%%%%%

\begin{thebibliography}{99}
%%%%%%%%%%%%%%%%%%%%%%%%%%%%%%
\bibitem{Bayart-Matheron}
{F. Bayart and \'E. Matheron},	%authors
\textit{Dynamics of Linear Operators}, %book
{Cambridge University Press},	%publ
{Cambridge},	%publaddr
{2009}.		%year
%%%%%%%%%%%%%%%%%%%
\bibitem{B-Ch-S2001}
{J. B\`es, K.C. Chan, and S.M. Seubert},	%authors
\textit{Chaotic unbounded differentiation operators},	%paper
{Integral Equations Operator Theory}	%journal
\textbf{40}	%vol
{(2001)},	%year
{no.~3},	%issue
{257--267}.	%pages
%%%%%%%%%%%%%%%%%%%%%%%%%%%%%%%%%%%%%%%%%
\bibitem{deL-E-G-E2003}
{R. deLaubenfels, H. Emamirad, and K.-G. Grosse-Erdmann},	%authors
\textit{Chaos for semigroups of unbounded operators},	%paper
{Math. Nachr.}	%journal
\textbf{261/262}	%vol
{(2003)},	%year
%{no.~3},	%issue
{47--59}.	%pages
%%%%%%%%%%%%%%%%%%%%%%%%%%%%%%
\bibitem{Devaney}
{R.L. Devaney},	%authors
\textit{An Introduction to Chaotic Dynamical Systems}, %book
{2nd ed.},	%bookinfo
{Addison-Wesley},	%publ
{New York},	%publaddr
{1989}.		%year
%%%%%%%%%%%%%%%%%%%%%%%%%%%%%%%
\bibitem{Dun-SchI}
{N. Dunford and J.T. Schwartz with the assistance of W.G. Bade and R.G. Bartle},	%authors
\textit{Linear Operators. Part I: General Theory},	%book
%{Pure and Applied Mathematics, vol. 7},	%bookinfo
{Interscience Publishers},	%publ
{New York},		%publaddr
{1958}.		%year
%%%%%%%%%%%%%%%%%%%%%%%%%%%%%%%%%%%%%%%%%
\bibitem{Gethner-Shapiro1987}
{R.M. Gethner and J.H. Shapiro},	%authors
\textit{Universal vector for operators on spaces of holomorphic functions},	%paper
{Proc. Amer. Math. Soc.}	%journal
\textbf{100}	%vol
{(1987)},	%year
{no.~2},	%issue
{281--288}.	%pages
%%%%%%%%%%%%%%%%%%%%%%%%%%%%%%%%%%%%%%%%%
\bibitem{Godefroy-Shapiro1991}
{G. Godefroy and J.H. Shapiro},	%authors
\textit{Operators with dense, invariant, cyclic vector manifolds},	%paper
{J. Funct. Anal.}	%journal
\textbf{98}	%vol
{(1991)},	%year
%{no.~3},	%issue
{229--269}.	%pages
%%%%%%%%%%%%%%%%%%%%%%%%%%%%%%%%%%%%%%%%
\bibitem{Grosse-Erdmann2000}
{K.-G. Grosse-Erdmann},	%authors
\textit{Hypercyclic and chaotic weighted shifts},	%paper
{Studia Math.}	%journal
\textbf{139}	%vol
{(2000)},	%year
{no.~1},	%issue
{47--68}.	%pages
%%%%%%%%%%%%%%%%%%%%%%%%%%%%%%%%%%%%%%%%
\bibitem{Grosse-Erdmann-Manguillot}
{K.-G. Grosse-Erdmann and A.P. Manguillot},	%authors
\textit{Linear Chaos}, %book
{Universitext},	%bookinfo
{Springer-Verlag},	%publ
{London},		%publaddr
{2011}.		%year
%%%%%%%%%%%%%%%%%%%%%%%%%%%%%%%%%%%%%%%%%%%%%
\bibitem{Kitai1982}
{C. Kitai},	%authors
\textit{Invariant Closed Sets for Linear Operators},	%paper
{Ph.D. Thesis},	%book
{University of Toronto},	%bookinfo
{1982}.		%year	
%%%%%%%%%%%%%%%%%%%%%%%%%%%%%%%%%%%%%%%%%%%%%
%\bibitem{MacLane}
%{G.R. MacLane},	%authors
%\textit{Sequences of derivatives and normal families},	%paper
%{J. Analyse Math.}	%journal
%\textbf{2}	%vol
%{(1952/53)},	%year
%{no.~3},	%issue
%{72--87}.	%pages
%%%%%%%%%%%%%%%%%%%%%%%%%%%%%%%%%%%%%%%%
\bibitem{arXiv:1811.06640}
{M.V. Markin},	%authors
\textit{On the chaoticity and spectral structure of Rolewicz-type unbounded operators},	%paper
%{Ibid.}	%journal
%\textbf{23}	%vol
%{(2017)},	%year
%{no.~3},	%issue
%arXiv:1811.06640.	%pages
\href{https://arxiv.org/abs/1811.06640v5}{\textcolor{blue}{arXiv:1811.06640}}.
%%%%%%%%%%%%%%%%%%%%%%%%%%%%%%%%%%%%%%%%
\bibitem{arXiv:2106.09682}
{M.V. Markin},	%authors
\textit{On the chaoticity of derivatives},	%paper
%{Methods Funct. Anal. Topology}, %journal
%\textbf{22}	%vol
%{(2016)},	%year
%{no.~2},	%issue
%{169--183}.	%pages
\href{https://arxiv.org/abs/2106.09682}{\textcolor{blue}{arXiv:2106.09682}}.
%%%%%%%%%%%%%%%%%%%%%%%%%%%%%%%%%%%%%%%%
\bibitem{arXiv:2106.14872}
{M.V. Markin},	%authors
\textit{On sufficient  and necessary conditions for linear hypercyclicity and chaos},	%paper
%{Methods Funct. Anal. Topology}, %journal
%\textbf{22}	%vol
%{(2016)},	%year
%{no.~2},	%issue
%{169--183}.	%pages
\href{https://arxiv.org/abs/2106.14872}{\textcolor{blue}{arXiv:2106.14872}}.
%%%%%%%%%%%%%%%%%%%%%%%%%%%%%%%%%%%%%%%%
\bibitem{Markin2026EOT}
{M.V. Markin},	%authors
\textit{Elementary Operator Theory: Abstract Spaces, Linear Operators, Fundamentals of Spectral Theory}, %book
{2nd ed., Revised and Upgraded},	%bookinfo
{De Gruyter Graduate},	%bookinfo
{De Gryuter Brill GmbH},	%publ
{Berlin/Boston}, %publaddr
{2026},					  %year
{ISBN 978-3-11-224776-1}. %ISBN
%%%%%%%%%%%%%%%%%%%%%%%%%%%%%%%%%%%%%%%%
\bibitem{Markin2027EFA}
{M.V. Markin},	%authors
\textit{Elementary Functional Analysis: Abstract Spaces, Fundamental Principles, Duality and Reflexivity}, %book
{2nd ed., Revised and Upgraded},	%bookinfo
{De Gruyter Graduate},	%bookinfo
{De Gryuter Brill GmbH},	%publ
{Berlin/Boston}, %publaddr
{2027},					  %year
{ISBN 978-3-11-224779-2}. %ISBN
%%%%%%%%%%%%%%%%%%%%%%%%%%%%%%%%%%%%%%%%
%\bibitem{MarkM-LSich2022}
%{M.V. Markin, G. Martinez Lazaro, and E.S. Sichel},	%authors
%\textit{On linear chaos in the spaces of vanishing and convergent sequences},	%paper
%{Methods Funct. Anal. Topology}, %journal
%\textbf{22}	%vol
%{(2016)},	%year
%{no.~2},	%issue
%{169--183}.	%pages
%\href{https://arxiv.org/abs/2203.02032}{\textcolor{blue}{arXiv:2203.02032}}.
%%%%%%%%%%%%%%%%%%%%%%%%%%%%%%%%%%%%%%%%%
\bibitem{arXiv:2209.04515}
{M.V. Markin and E. Montoya},	%authors
\textit{On hypercyclicity and linear chaos in a nonclassical sequence space and beyond},	%paper
%{Methods Funct. Anal. Topology}, %journal
%\textbf{22}	%vol
%{(2016)},	%year
%{no.~2},	%issue
%{169--183}.	%pages
\href{https://doi.org/10.48550/arXiv.2209.04515}{\textcolor{blue}{arXiv:2209.04515}}.
%%%%%%%%%%%%%%%%%%%%%%%%%%%%%%%%%%%%%%%%%
\bibitem{MarkSogh2021}
{M.V. Markin and O.B. Soghomonian},	%authors
\textit{On a characterization of convergence in Banach spaces with a Schauder basis},	%paper
{Int. J. Math. Math. Sci.}	%journal
\textbf{2021}	%vol
{(2021)},	%year
{Article ID 1640183}, %issue
{5 pp.} %pages
%%%%%%%%%%%%%%%%%%%%%%%%%%%%%%%%%%%%%%%%
\bibitem{Rolewicz}
{S. Rolewicz},	%authors
\textit{On orbits of elements},	%paper
{Studia Math.}	%journal
\textbf{32}	%vol
{(1969)},	%year
%{no.~9},	%issue
{17--22}.	%pages
%%%%%%%%%%%%%%%%%%%%%%%%%%%%%%%%%%%%%%%%%%%%%%
\end{thebibliography}
\end{document}